\documentclass[11pt,reqno]{amsart}
\usepackage[top=1.2in, bottom=1.2in, left=1.2in, right=1.2in]{geometry}
\usepackage{amsfonts, amssymb, amsmath, mathtools, dsfont, xcolor}%, enumerate, bm, xspace, graphicx, caption, subcaption}

\usepackage[linktoc=page, colorlinks, linkcolor=blue, citecolor=blue]{hyperref}
\usepackage{enumerate, commath}
\usepackage[utf8]{inputenc}

\usepackage{graphicx}

\usepackage{subcaption} 
\usepackage[font=small]{caption} 
\numberwithin{equation}{section}

\DeclareMathOperator{\E}{\mathbb{E}}

\DeclareMathOperator{\Var}{Var}

\DeclareMathOperator{\sign}{sign}

\DeclareMathOperator{\rank}{rank}

\newcommand{\ip}[2]{\langle#1,#2\rangle}

\def \N {\mathbb{N}}

\def \R {\mathbb{R}}

\def \FF {\mathcal{F}}

\def \s {\sigma}

\def \tran {\mathsf{T}}

\def \psitwo {{\psi_2}}

\def\eps{{\varepsilon}}

\newtheorem{theorem}{Theorem}[section]
\newtheorem{proposition}[theorem]{Proposition}
\newtheorem{corollary}[theorem]{Corollary}
\newtheorem{lemma}[theorem]{Lemma}

\theoremstyle{remark}
\newtheorem{remark}[theorem]{Remark}

\begin{document}

\title{Covariance loss,  Szemeredi regularity, \\ and differential privacy}

%\author{March Boedihardjo, Thomas Strohmer, and Roman Vershynin}
\author{March Boedihardjo}
\address{Department of Mathematics, ETH Z\"urich}
\email{march.boedihardjo@ifor.math.ethz.ch}
\author{Thomas Strohmer}
\address{Department of Mathematics, University of California, Davis and Center of Data Science and Artificial Intelligence Research, UC Davis}
\email{strohmer@math.ucdavis.edu}
\author{Roman Vershynin}
\address{Department of Mathematics, University of California, Irvine}
\email{rvershyn@uci.edu}

\maketitle

\begin{abstract}
We show how randomized rounding based on Grothendieck's identity can be used to prove a nearly tight bound on the covariance loss--the amount of covariance that is lost by taking conditional expectation.
This result yields a new type of weak Szemeredi regularity lemma
for positive semidefinite matrices and kernels.
Moreover, it can be used to construct differentially private synthetic data.
\end{abstract}

\section{Introduction}
%----------
Consider  a random vector  $X$ taking values in $\R^d$ and a $\s$-algebra  $\FF$.
Let $Y$ denote the conditional expectation: $Y = \E[X|\FF]$. 
In dimension $d=1$, where $X$ and $Y$ are random variables, the law of total variance states that
\begin{equation}	\label{eq: ltv}
\Var(X)-\Var(Y) = \E X^2 - \E Y^2 = \E (X-Y)^2 \ge 0.
\end{equation}
Thus, taking conditional expectation underestimates the variance. This observation extends to higher dimensions. Namely, let $\Sigma_X = \E (X-\E X)(X-\E X)^\tran$ denote the covariance matrix of $X$, 
and similarly for $\Sigma_Y$. Then
\begin{equation}\label{covloss}
\Sigma_X - \Sigma_Y = \E XX^\tran - \E YY^\tran = \E (X-Y)(X-Y)^\tran \succeq 0,
\end{equation}
where $\succeq$ denotes the Loewner order, in which $A \succeq B$ if $A-B$ is positive semidefinite.
Just like in the one-dimensional case, we see that taking conditional expectation underestimates the covariance. 

In~\cite{BSV2021a}, we asked the basic question: ``How big is the {\em covariance loss} $\Sigma_X - \Sigma_Y$?''
The answer will obviously depend on the choice of the sigma-algebra $\FF$, prompting the next question:  ``What sigma-algebra $\FF$ of given complexity minimizes the covariance loss?''

It was shown in~\cite{BSV2021a}  that there exists a partition of the sample space into at most $k$ sets
  such that for the sigma-algebra $\FF$ generated by this partition, the covariance loss~\eqref{covloss} is upper bounded by  $C \sqrt{\log \log k/\log k}$, where $C$ is an absolute constant. In~\cite{jain2022optimal}, Jain, Sah, and Sawhney were able to improve this bound to $C/\sqrt{\log k}$. It follows from Proposition 3.14 in~\cite{BSV2021a} that  this bound is optimal up to the value of the constant $C$. The proofs of the bounds in both~\cite{BSV2021a} and~\cite{jain2022optimal} are somewhat technical. In this paper we give a new approach to this problem, which is conceptually related to Szemeredi regularity. This approach yields a tighter bound in terms of the constant on the covariance loss (Corollary~\ref{cor: covariance loss}), as well as a much shorter and more elementary proof. 

The celebrated Szemeredi regularity lemma states that that for every large enough graph, the set of nodes can be divided into subsets of about the same size so that the edges between different subsets behave almost randomly~\cite{szemeredi1975regular}. Expressed in the language of linear algebra, it says that the adjacency matrix of a graph can be approximated by a sum of cut matrices\footnote{A cut matrix is a rank-one matrix which is constant on a block and zero elsewhere~\cite{frieze1999quick}.}.
The Szemeredi regularity lemma, both in the combinatorial and linear-algebraic forms, has many deep algorithmic and combinatorial applications, see e.g.~\cite{frieze1999quick, lovasz2007szemeredi,komlos2000regularity,bodwin2022unified}. 
We prove a new version of (weak) Szemeredi regularity lemma, 
which states that any positive semidefinite kernel $K(x,y)$ whose diagonal is uniformly bounded can be decomposed into a sum of $k$ step-functions plus an error term whose $L^2$-norm is $O(1/\sqrt{\log k})$.

\medskip
The outline of the paper is as follows.  
We will first bound the covariance loss by the ``covariance increment'' (Proposition~\ref{prop: covariance increment}), find a nice representation of the covariance increment (Proposition~\ref{prop: increment ip}), and combine it with randomized rounding based on Grothendieck's identity to bound the covariance increment (Theorem~\ref{thm: regularity}). Thus we can not only bound the covariance loss (Corollary~\ref{cor: covariance loss}) but also derive a new type of weak Szemeredi regularity for matrices (Theorem~\ref{thm: wsr}) and kernels (Theorem~\ref{thm: wsr kernels}). 
We conclude by illustrating how our results can be used in connection with differential privacy and synthetic data to improve the accuracy bounds obtained in~\cite{BSV2021a}.

\subsection{Notation}
The subgaussian norm of a random variable $X$ is denoted by
$\|X\|_{\psitwo}$ (see e.g.~\cite{vershyninbook}), $\E(X)$ is the expectation of $X$ and $\|X\|_{L^{p}}=(\mathbb{E}|X|^{p})^{1/p}$. For $d\times d$ matrices $A,B$, define the inner product $\ip{A}{B}=\mathrm{Tr}(AB^{T})$ and the Frobenius norm $\| A \|_F=\sqrt{\ip{A}{A}}$.

\section{The covariance increment}
%-----------------

We start out by  bounding the covariance loss $\E XX^\tran - \E YY^\tran$ by the covariance increment.

In the sequel, $X$ is a random vector taking values in $\R^d$, $\FF$ is a $\s$-algebra and $Y = \E[X|\FF]$. Moreover, $(X',\mathcal{F}',Y')$ is an independent copy of $(X,\mathcal{F},Y)$ so that the sample space is a product space for which $X,\FF,Y$ are based on the first component of the product space and $X',\mathcal{F}',Y'$ are based on the second.

With a slight abuse of notation, the product $\s$-algebra $\FF\times\FF$ denotes the $\s$-algebra generated by $\mathcal{F}$ and $\mathcal{F}'$. Moreover, the inner product $\ip{Y}{Y'}=\mathbb{E}(\ip{X}{X'}|\,\FF\times\FF)$.

\begin{proposition}[Covariance loss and covariance increment]		\label{prop: covariance increment}
\begin{equation}	\label{eq: covariance increment}
\norm{\E XX^\tran - \E YY^\tran}_F^2
\le \norm{\E XX^\tran}_F^2 - \norm{\E YY^\tran}_F^2.
\end{equation}
\end{proposition}

The proof of Proposition~\ref{prop: covariance increment} will be based on the following lemma:

\begin{lemma}				\label{lem: frobenius by ip}
  Let $(U,V)$ be a pair of random vectors taking values in $\mathbb{R}^{d}$, and let $(U',V')$ be an independent copy.
  Then 
  $$
  \ip{\E UU^\tran}{\E VV^\tran} = \E \ip{U}{V'}^2.
  $$ 
  In particular, setting $V=U$, we have
  $$
  \norm{\E UU^\tran}_F^2 = \E \ip{U}{U'}^2.
  $$
\end{lemma}

\begin{proof}
The first identity readily follows if we first use the identical distribution, and then independence:
$$
\ip{\E UU^\tran}{\E VV^\tran}
= \ip{\E UU^\tran}{\E V'(V')^\tran}
= \E \ip{UU^\tran}{V'(V')^\tran}
= \E \ip{U}{V'}^2.
$$
The lemma is proved.
\end{proof}

\medskip

\begin{proof}[Proof of Proposition~\ref{prop: covariance increment}]
Expanding the square of the Frobenius norm, we can express the left hand side of \eqref{eq: covariance increment} as 
$\norm[1]{\E XX^\tran}_F^2  - 2\ip{\E XX^\tran}{\E YY^\tran} + \norm[0]{\E YY^\tran}_F^2$. 
After simplification, we see that inequality \eqref{eq: covariance increment} is equivalent to 
$$
\norm{\E YY^\tran}_F^2 
\le \ip{\E XX^\tran}{\E YY^\tran}.
$$
Using Lemma~\ref{lem: frobenius by ip} we can rewrite this as 
\begin{equation}	\label{eq: YY' XY'}
\E \ip{Y}{Y'}^2
\le \E \ip{X}{Y'}^2,
\end{equation}
To check \eqref{eq: YY' XY'}, recall that $Y = \E[X|\FF]$ and apply the conditional Jensen's inequality. 
\end{proof}

The covariance increment has a nice representation, which will come handy in our further analysis.

\begin{proposition}		\label{prop: increment ip}
\begin{equation}	\label{eq: increment}
\norm{\E XX^\tran}_F^2 - \norm{\E YY^\tran}_F^2
= \E \Big( \ip{X}{X'} - \ip{Y}{Y'} \Big)^2.
\end{equation}
\end{proposition}

\begin{proof}
By Lemma~\ref{lem: frobenius by ip}, the left hand side of \eqref{eq: increment} equals
$\E \ip{X}{X'}^2 - \E \ip{Y}{Y'}^2$.
Note that $\ip{Y}{Y'}=\E(\ip{X}{X'}|\,\FF\times\FF)$.
To finish the proof, apply the law of total variance \eqref{eq: ltv} for $\ip{X}{X'}$ instead of $X$
and $\ip{Y}{Y'}$ instead of $Y$.
\end{proof}

\section{Bounding the covariance increment}
%--------------

\begin{theorem}		\label{thm: regularity}
Let $X$ be a random vector taking values $\R^d$ such that $\norm{X}_2 \le 1$ a.s.
Then, for any $r \in \N$, there exists a partition of the sample space into at most $2^r$ parts
such that for the $\s$-algebra $\FF$ generated by this partition, the conditional expectation
$Y = \E[X|\FF]$ satisfies  
$$
\E \Big( \ip{X}{X'} - \ip{Y}{Y'} \Big)^2
\le \frac{\pi^2}{r}.
$$
\end{theorem}

The proof of Theorem~\ref{thm: regularity} will utilize Grothendieck's identity (see e.g. \cite[Lemma~3.6.6]{vershyninbook}):

\begin{lemma}[Grothendieck's identity]\label{groid}
 Let $x, x'$ be a pair of unit vectors in $\R^d$, and let $g \sim N(0,I_d)$. 
 Then 
 $$
 \ip{x}{x'} = \sin \left[ \frac{\pi}{2} \E \sign \ip{x}{g} \sign \ip{x'}{g} \right].
 $$
\end{lemma}

\begin{proof}[Proof of Theorem~\ref{thm: regularity}]
{\em Step 1.} Let us first make a stronger assumption, namely that $\norm{X}_2=1$ a.s., 
and prove a weaker conclusion, namely that there exists a random variable $Z$ measurable with respect to the product $\s$-algebra $\FF \times \FF$ and such that 
\begin{equation}	\label{eq: Z}
\norm{\ip{X}{X'} - Z}_{L^2}
\le \frac{\pi}{2\sqrt{r}}.
\end{equation}
To this end, consider independent random vectors $g_1,\ldots,g_k \in N(0,I_d)$,
and for $x,x'\in\mathbb{R}^{d}$, define the random variable
$$
F_g(x,x') = \sin \left[ \frac{\pi}{2} \cdot \frac{1}{r} \sum_{k=1}^r \sign \ip{x}{g_k} \sign \ip{x'}{g_k} \right].
$$
Denoting $\xi_k = \sign \ip{x}{g_k} \sign \ip{x'}{g_k}$, noting that the function $\sin(\cdot)$ is $1$-Lipschitz and applying Lemma \ref{groid}, we obtain 
\begin{equation}	\label{eq: xx' Fxx'}
\abs{\ip{x}{x'}-F_g(x,x')}
\le \frac{\pi}{2} \cdot \frac{1}{r} \abs{\sum_{k=1}^r (\xi_k-\E \xi_k)}.
\end{equation}
By independence, this yields
$$
\E_g \left( \ip{x}{x'}-F_g(x,x') \right)^2
\le \left( \frac{\pi}{2r} \right)^2 \cdot r \Var(\xi_1)
\le \frac{\pi^2}{4r}.
$$
Substitute $x=X$, $x'=X'$ and take expectation with respect to $X$ and $X'$. 
By Fubini theorem, there exists a realization of the random vectors $g_1,\ldots,g_k$ such that 
$$
\E \left( \ip{X}{X'}-F_g(X,X') \right)^2
\le \frac{\pi^2}{4r},
$$
Fix such a realization. 
Let $\FF = \s(V(X))$ be the $\s$-algebra generated by the random vector $V(X) = (\sign \ip{x}{g_k})_{k=1}^r$. Since $V(X)$ takes at most $2^r$ values, $\FF$ satisfies the requirement of the theorem. Moreover, the random vector 
$$
Z = F_g(X,X') = \sin \left( \frac{\pi}{2r} \ip{V(X)}{V(X')} \right)
$$
is measurable with respect to the product sigma-algebra $\FF \times \FF$. 
Thus, we proved \eqref{eq: Z}. 

\medskip

{\em Step 2: replacing $Z$ with $\ip{Y}{Y'}$.}
The conditional expectation $\E \left[ \cdot|\,\FF \times \FF \right]$ is an orthogonal projection in $L^2$
onto the subspace of random variables that are $\FF \times \FF$-measurable. 
Since the random variable $Z$ constructed in the previous step is $\FF \times \FF$-measurable, 
and $\ip{Y}{Y'} = \E \left[ \ip{X}{X'}|\,\FF \times \FF \right]$, it follows from \eqref{eq: Z} that 
$$
\norm{\ip{X}{X'} - \ip{Y}{Y'}}_{L^2}
\le \norm{\ip{X}{X'} - Z}_{L^2}
\le \frac{\pi}{2\sqrt{r}}.
$$

{\em Step 3: removing the unit norm requirement.}
We proved the theorem under the additional assumption that $X$ is a unit random vector.
Now let $X$ be an arbitrary random vector satisfying $\norm{X}_2 \le 1$.
Let $N \in \N$, and consider the random vector 
$W = \sqrt{1-\norm{X}_2^2} \cdot \theta$,
where $\theta$ is a random vector that is 
uniformly distributed in the unit basis $\{e_1,\ldots,e_N\}$ of $\R^N$.
Then the direct sum $X \oplus W = (X_1,\ldots,X_d, W_1,\ldots,W_N)$ is a unit random vector 
in $\R^{d+N}$. Applying the previous step for $X \oplus W$, we find that
\begin{align} 
\frac{\pi}{2\sqrt{r}}
  &\ge \norm{\ip{X \oplus W}{X' \oplus W'} - \ip{Y \oplus \eta}{Y' \oplus \eta'}}_{L^2}
	\quad \text{(where $\eta = \E(W|\FF)$)} \nonumber\\
  &= \norm{\ip{X}{X'} - \ip{Y}{Y'} + \ip{W}{W'} - \ip{\eta}{\eta'}}_{L^2} \nonumber\\
  &\ge \norm{\ip{X}{X'} - \ip{Y}{Y'}}_{L^2} - \norm{\ip{W}{W'} - \ip{\eta}{\eta'}}_{L^2}. \label{eq: direct sum}
\end{align}
Since $\ip{\eta}{\eta'} = \E(\ip{W}{W'}|\,\FF\times\FF)$, 
the law of total variance \eqref{eq: ltv} and definition of $W$ yields
\begin{equation}	\label{eq: WW'}
\norm{\ip{W}{W'} - \ip{\eta}{\eta'}}_{L^2}
\le \norm{\ip{W}{W'}}_{L^2}
\le \norm{\ip{\theta}{\theta'}}_{L^2}
= \frac{1}{\sqrt{N}},
\end{equation}
where the last identity follows since $\ip{\theta}{\theta'}$ is Bernoulli with parameter $1/N$.
Taking $N$ to be large enough, we conclude from \eqref{eq: direct sum} that 
$$
\norm{\ip{X}{X'} - \ip{Y}{Y'}}_{L^2} \le \frac{\pi}{\sqrt{r}}.
$$
Theorem~\ref{thm: regularity} is proved.
\end{proof}
\begin{remark}\label{thm: regularityrem}
In the conclusion of Theorem \ref{thm: regularity}, we can replace the constant $\pi^{2}$ by $\pi^{2}/4+\epsilon$ for any fixed $\epsilon>0$.
\end{remark}

\section{Subgaussian error}
%............

We can extend the bound in Theorem~\ref{thm: regularity} to the subgaussian norm, 
which is the Orlicz norm with respect to the Young function $\psitwo(x)=e^{x^2}-1$. 
Thus, a random variable $X$ is subgaussian if 
\begin{equation}	\label{eq: subgaussian}
\norm{X}_\psitwo = \inf\{t>0:\; \E \psitwo(X/t) \le 1\} < \infty,
\end{equation}
and this quantity is called the subgaussian norm of $X$, 
see \cite[Sections~2.5, 2.7.1]{vershyninbook}.

\begin{theorem}		\label{thm: regularity subgaussian}
Let $X$ be a random vector taking values $\R^d$ such that $\norm{X}_2 \le 1$ a.s.
Then, for any $r \in \N$, there exists a partition of the sample space into at most $2^r$ parts
such that for the $\s$-algebra $\FF$ generated by this partition, the conditional expectation
$Y = \E[X|\FF]$ satisfies  
\begin{equation}	\label{eq: regularity subgaussian}
\norm{\ip{X}{X'} - \ip{Y}{Y'}}_\psitwo
\le \frac{C}{\sqrt{r}}.
\end{equation}
\end{theorem}

We can prove this result by modifying the proof of Theorem~\ref{thm: regularity}.
Let us explain how to do this. 

In Step~1, instead of using additivity of variance, we can use Hoeffding's inequality to control the sum of independent Bernoulli random variables $\xi_k$. Denoting $\zeta_k = \xi_k-\E\xi_k$, we have 
$\norm{\sum_{k=1}^r \zeta_k}_\psitwo \lesssim \sqrt{r}$, see \cite[Proposition~2.6.1]{vershyninbook}. Using this in \eqref{eq: xx' Fxx'}, we get 
$$\norm{\ip{x}{x'}-F_g(x,x')}_\psitwo
\lesssim 1/\sqrt{r}.$$
By definition of the subgaussian norm, this means that 
$$\E_g \psitwo\left( (\ip{x}{x'}-F_g(x,x')) c\sqrt{r} \right) \le 1,$$ where $c>0$ is some absolute constant.
Substituting here $x=X$ and $x'=X'$ and applying the Fubini inequality as in Step~1, 
we obtain the following version of \eqref{eq: Z}:
$$
\norm{\ip{X}{X'} - Z}_\psitwo
\lesssim \frac{1}{\sqrt{r}}.
$$

In Step~2, although the conditional expectation is not a metric projection in the subgaussian norm,
it is an approximate metric projection:

\begin{lemma}		\label{lem: centering}
  Let $X$ be a random variable and $\FF$ be a $\s$-algebra. 
  Then, for any random variable $Z$ that is $\FF$-measurable, and any $p \ge 1$, we have
  $$
  \norm{X - \E(X|\FF)}_\psitwo \le 2 \norm{X-Z}_\psitwo.
  $$
\end{lemma}

\begin{proof}
Subtracting and adding $Z$ and using triangle inequality, we get
$$
\norm{X - \E(X|\FF)}_\psitwo 
\le \norm{X-Z}_\psitwo + \norm{\E(X|\FF)-Z}_\psitwo.
$$
Since $Z$ is $\FF$-measurable, we have
$$
\norm{\E(X|\FF)-Z}_\psitwo
= \norm{\E(X-Z|\FF)}_\psitwo
\le \norm{X-Z}_\psitwo,
$$
where the last step follows from definition of subgaussian norm \eqref{eq: subgaussian}
and conditional Jensen's inequality.
Combine the two bounds to complete the proof.
\end{proof}
 
Using this lemma for $\ip{X}{X'}$ instead of $X$, we obtain in Step~2 that 
$$
\norm{\ip{X}{X'} - \ip{Y}{Y'}}_\psitwo
\le 2 \norm{\ip{X}{X'} - Z}_\psitwo
\lesssim \frac{1}{\sqrt{r}}.
$$

In Step~3, we argue in a similar way about the subgaussian norm. 
The bound \eqref{eq: WW'} becomes
$$
\norm{\ip{W}{W'} - \ip{\eta}{\eta'}}_\psitwo
\le 2\norm{\ip{W}{W'}}_\psitwo
\le \norm{\ip{\theta}{\theta'}}_\psitwo
\lesssim \frac{1}{\sqrt{\log N}},
$$
Here in the first step we use Lemma~\ref{lem: centering} for $X = \ip{W}{W'}$ and $Z=0$, 
and the last step is a straightforward bound on the subgaussian norm of 
a Bernoulli random variable with parameter $1/N$.
This completes the proof.

\section{Implications}
%--------------------

\subsection{Covariance loss}
%.................
Combining Proposition~\ref{prop: covariance increment}, Proposition~\ref{prop: increment ip} 
and Theorem~\ref{thm: regularity}, we obtain:

\begin{corollary}[Covariance loss]		\label{cor: covariance loss}
Let $X$ be a random vector taking values $\R^d$ such that $\norm{X}_2 \le 1$ a.s.
Then, for any $r \in \N$, there exists a partition of the sample space into at most $2^r$ parts
such that for the $\s$-algebra $\FF$ generated by this partition, the conditional expectation
$Y = \E[X|\FF]$ satisfies  
$$
\norm{\E XX^\tran - \E YY^\tran}_F
\le \frac{\pi}{\sqrt{r}}.
$$
\end{corollary}

In particular, for any $k$ with $\log_{2}k\in \N$, there exists a partition of the sample space into at most $k$ parts
such that for the $\s$-algebra $\FF$ generated by this partition, we have
\begin{equation}\label{covloss2}
\norm{\E XX^\tran - \E YY^\tran}_F
\le \frac{\pi}{\sqrt{\log_2 k}}.
\end{equation}

\begin{remark}[Optimality]
  Proposition 3.14 in~\cite{BSV2021a} implies that there exists a random vector $X$ taking values in the   unit ball of $\mathbb{R}^{d}$ such $\norm{\E XX^\tran - \E YY^\tran}_F\geq\frac{\sqrt{\log_2(e)}}{80\sqrt{\log_2 k}}$ for any  $\s$-algebra $\FF$ generated by a partition into at most $k$ parts. 
Thus the bound in~\eqref{covloss2} is sharp up to an absolute constant.
\end{remark}

\subsection{Weak Szemeredi regularity}
%..............

\begin{theorem}[Weak Szemeredi regularity]			\label{thm: wsr}
  Let $A$ be an $n \times n$ positive semidefinite matrix such that $A_{ii} \le 1$ for all $i$.
  Then, for any $r \in \N$, there exists a partition 
  $[n] = I_1 \cup \ldots \cup I_k$ with $k \le 2^r$,  
  and a matrix $B$ that is constant on each block $I_i \times I_j$ and such that 
  \begin{equation}\label{frobnorm}
  \frac{1}{n} \norm{A-B}_F 
  \le \frac{\pi}{\sqrt{r}}.
  \end{equation}
  Moreover, $B$ can be computed by averaging the entries of $A$ in each block.
\end{theorem}

\begin{proof}
By assumption, $A$ can be represented as the Gram matrix of some vectors $x_1,\ldots,x_n$ satisfying $\norm{x_i}_2 \le 1$ for all $i$. Thus $A = [\ip{x_i}{x_j}]_{i,j=1}^n$. 
Apply Theorem~\ref{thm: regularity} for the random vector $X$ that is uniformly distributed on $\{x_1,\ldots,x_n\}$ to complete the proof.
\end{proof}

\begin{remark}

Rewriting the approximation error in Theorem~\ref{thm: wsr} as 
$$
\frac{1}{n^2} \sum_{i,j=1}^n (A_{ij}-B_{ij})^2 \le \frac{\pi^2}{r},
$$
we can interpret it as a bound on the {\em mean squared error} of the entries.
\end{remark}

\begin{remark}

Applying Theorem~\ref{thm: regularity subgaussian} instead of Theorem~\ref{thm: regularity}, we can replace the Frobenius norm in~\eqref{frobnorm} by stronger matrix norms, such as the $\ell_p$ norm of the entries.
\end{remark}

%Mercer's theorem in combination with Theorem~\ref{thm: regularity subgaussian} should yield a continuous analog of Theorem~\ref{thm: wsr}, something like this: 

\if 0
\begin{theorem}[Weak Szemeredi regularity for kernels]			\label{thm: wsr kernels}	
  Let $\Omega$ be a closed subset of $\R^d$, 
  $\mu$ be a Borel probability measure on $\Omega$, and 
  $K(x,y): \Omega \times \Omega \to \R$ be a continuous positive semidefinite kernel 
  satisfying $\norm{K}_{L^2(\mu \times \mu)} < \infty$ and $K(x,x) \le 1$ for all $x$.  
  Then, for any $r \in \N$, there exists a partition 
  $\Omega = I_1 \cup \ldots \cup I_k$ with $ \le 2^r$,  
  and a function $L(x,y)$ that is constant on each block $I_i \times I_j$ and such that 
  $$
  \norm{K-L}_{L^2(\mu \times \mu)}
  \le \frac{\pi}{2 \sqrt{r}}.
  $$
  Moreover, $L$ can be computed by averaging $K$ on each block.
\end{theorem}
\fi

\begin{theorem}[Weak Szemeredi regularity, analytic form] \label{thm: wsr kernels}	
Let $(\Omega,\mu)$ be a probability measure space. Let $K:\Omega\times\Omega\to\mathbb{R}$ be a measurable positive definite kernel such that $K(t,t)\leq 1$ for all $t\in\Omega$. Then for any $r\in\mathbb{N}$, there exists a partition $\Omega=I_{1}\cup\ldots\cup I_{k}$ with $k\leq 2^{r}$, and a function $L:\Omega\times\Omega\to\mathbb{R}$ that is constant on each block $I_{i}\times I_{j}$ such that
%\[\|K-L\|_{L^{2}(\mu\times\mu)}\leq\frac{\pi}{2\sqrt{r}}+\gamma.\]
\begin{equation}\label{kernel}
\|K-L\|_{L^{2}(\mu\times\mu)}\leq\frac{\pi}{\sqrt{r}}.
\end{equation}
Moreover, $L$ can be computed by averaging $K$ on each block.
\end{theorem}

\begin{proof}
When  $\Omega$ is finite, using Theorem~\ref{thm: regularity} (see also Remark \ref{thm: regularityrem}), one can easily establish
$\|K-L\|_{L^{2}(\mu\times\mu)}\leq\frac{\pi}{2\sqrt{r}}+\epsilon$ for any fixed $\epsilon>0$. 
The general case follows by a standard approximation argument, which we leave to the reader.
\end{proof}

%\textcolor{blue}{[RV: March, Thomas, thanks for the proof of the general case. At this point, I suggest to skip it and leave it to the reader, in order to keep the paper short and clear. If the referee complains, or we get questions from readers, we will have it handy and can add details.]}

\subsection{Comparison to the existing work on weak Szemeredi regularity}
%..........

Our result is more restrictive but stronger that the classical weak regularity lemma by Frieze and Kannan~\cite{frieze1999quick}. There the approximation error $R=K-L$ is measured in the {\em cut norm}
$$
\norm{R}_\square 
= \sup_{S,T \subset \Omega} \abs{\int_{S \times T} R(x,y) \, d\mu(x) d\mu(y)}.
$$
The cut norm is equivalent to the operator norm $L^1 \to L^\infty$ (see e.g. 
\cite{lovasz2007szemeredi}) and clearly satisfies $\norm{R}_\square \le \norm{R}_{L^2(\mu \times \mu)}$. So the Hilbert-Schmidt norm, which is the focus of the current paper, 
is {\em stronger} than the cut norm. 
As a result, our bound \eqref{kernel} automatically extends to the cut norm.
A similar bound for the cut norm was established in the original work of Frieze and Kannan~\cite{frieze1999quick} in wider generality: it holds only for kernels but for an arbitrary measurable function $K(x,y)$ that is pointwise bounded by $1$, see \cite[Lemma~3.1]{lovasz2007szemeredi}. In contrast to this, positive semidefiniteness is required
for any nontrivial bound on the error in the stronger Hilbert-Schmidt norm, such as the one in Theorem~\ref{thm: wsr kernels}. We will see this in the next section.

A matrix decomposition in the spirit of Theorem~\ref{thm: wsr}, i.e. with error bounded in the Frobenius norm, appears in \cite[Theorem~7]{deshpande2012zero} by Deshpande, Kannan, and Srivastava. Their theorem, which does not require the matrix to be positive semidefinite, is only nontrivial for low-rank matrices $A$, namely for matrices whose rank is at most logarithmic\footnote{Theorem~7 in \cite{deshpande2012zero} approximates $A$ by a sum of roughly $t=\rank(A) \log^4 n$ cut-matrices; they can be further broken down into $2^t$ smaller cut-matrices with disjoint support. This ultimate decomposition is nontrivial only if $2^t<n^2$.} in the dimension $n$. Our results do not have a rank restriction.

\section{Optimality}
%============
The decay rate $1/r$ in the conclusion of Theorem \ref{thm: regularity} is optimal, because the decay rate $1/\sqrt{r}$ in the conclusion of Corollary \ref{cor: covariance loss} is optimal.

All previous results in the literature on weak Szemeredi regularity do not need the assumption that the matrix $A$ is positive semidefinite, see e.g.~\cite{frieze1999quick,lovasz2007szemeredi,bodwin2022unified}. Can it be removed from our Theorem~\ref{thm: wsr}? The following result says that it cannot be removed. 

\begin{proposition}[Positive semidefiniteness is essential]		\label{prop: hadamard}
  Let $A$ be an $n \times n$ Hadamard matrix, 
  and let $[n] = I_1 \cup \ldots \cup I_k$ be any partition. 
  Then, for any matrix $B$ that is constant on each block $I_i \times I_j$, we have
  $$
  \frac{1}{n} \norm{A-B}_F 
  \ge \frac{1}{2} \sqrt{1-\frac{k}{n}}.
  $$
\end{proposition}

Thus, the error does not vanish unless the number of parts $k$ is extremely large, namely $k=n-o(n)$. 

The proof will use the following elementary bound: 

\begin{lemma}		\label{lem: orthogonal difference}
  Let $x_1,\ldots,x_m$ are orthonormal vectors in $\R^n$. 
  Then for each $k=1,\ldots,m$ we have
  $$
  \norm[2]{x_k - \frac{1}{m} \sum_{t=1}^m x_t}_2 \ge 1-\frac{1}{m}.
  $$
\end{lemma}

\begin{proof}
The norm of the vector $x_k - \frac{1}{m} \sum_{t=1}^m x_t$ is bounded below by the inner product of that vector and $x_k$, which equals $1-1/m$.
\end{proof}

\begin{proof}[Proof of Proposition~\ref{prop: hadamard}]
The minimum of $\norm{A-B}_F$ over all matrices $B$ in the proposition can only get smaller if we minimize over all matrices that are constant on smaller blocks, namely on the sets $\{\ell\} \times I_j$, where $\ell=1,\ldots,n$ and $j=1,\ldots,k$. The latter minimum is attained\footnote{To see this, argue like in Step~2 of the proof of Theorem~\ref{thm: regularity}.}
for the matrix $B$ that is obtained by averaging the entries of $A$ in each block $\{\ell\} \times I_j$. Equivalently, $B$ 
is obtained by averaging the columns of $A$ in each set $I_j$. Formally, the columns of $B$ are
$$
B_\ell = \frac{1}{\abs{I_j}} \sum_{t \in I_j} A_t
\quad \text{for } \ell \in I_j,
$$ 
where $B_\ell$ is the $\ell$ column of $B$ and $A_t$ is the $t$th column of $A$. Note that there at least $n-k$ indices belong to the non-singleton blocks $I_j$, i.e.
$$
T \coloneqq \bigcup_{j:\; \abs{I_j}>1} I_j
\quad \text{satisfies}\quad 
\abs{T} \ge n-k.
$$
(Indeed, since $T^c$ is the union of the singleton blocks, its cardinality equals the number of such blocks $I_j$, which is bounded by the total number of blocks $k$.)

Pick any index $\ell \in T$, so $\ell \in I_j$ where $\abs{I_j} \ge 2$.
Since $A$ is an Hadamard matrix, its columns $A_t$ are orthogonal, and their Euclidean norms equal $\sqrt{n}$. Thus, applying Lemma~\ref{lem: orthogonal difference} and rescaling, we have
$$
\norm{A_\ell-B_\ell}_2 
\ge \Big( 1-\frac{1}{\abs{I_j}} \Big) \sqrt{n}
\ge \frac{\sqrt{n}}{2}.
$$
It follows that 
$$
\norm{A-B}_F^2 
= \sum_{\ell=1}^n \norm{A_\ell-B_\ell}_2^2
\ge \sum_{\ell \in T} \norm{A_\ell-B_\ell}_2^2
\ge \abs{T} \Big( \frac{\sqrt{n}}{2} \Big)^2
\ge \frac{n(n-k)}{4}.
$$
%The conclusion of the proposition follows. 
\end{proof}

\section{Application to differential privacy and synthetic data}

In this section we briefly describe how the findings in this paper can be applied to improve the results in~\cite{BSV2021a} on generating private synthetic data.
The partitioning result behind the covariance loss bound in Corollary~\ref{cor: covariance loss} and the Szemeredi regularity bound in 
Theorem~\ref{thm: wsr} provide a natural technique towards data privacy related to k-anonymity and differential privacy.
While the Szemeredi regularity lemma has been proposed in~\cite{foffano2019you,minello2020k} as a mechanism for data anonymization, the results in those papers are merely empirical and come without any theoretical guarantees.

The findings in this paper can be readily used to improve upon the data privacy and utility guarantees for k-anonymity and differential privacy in~\cite{BSV2021a}. 
For instance, using  Corollary~\ref{cor: covariance loss} in lieu of Theorem~1.2 in~\cite{BSV2021a}, but otherwise following the same procedure, we arrive at the following theorem which gives an improved accuracy guarantee compared to Theorem 5.14 in~\cite{BSV2021a}.
We leave the details to the reader and refer to~\cite{BSV2021a}  for the precise problem statement and a detailed description of the underlying concepts.

\begin{theorem}
 Let $K$ be a convex set in $\R^p$ that lies in the unit Euclidean ball $B_2^p$. Let $\alpha >0$. If $n,p \in \N$, $\eps > 0$ satisfy $n > C(\alpha)\frac{p}{\eps}$, then there is an $\eps$-differentially private algorithm with input $x_1,\ldots,x_n \in K$ and output $u_1,\ldots,u_{m(\alpha)} \in K$ such that
  $$
  \E \norm[3]{\frac{1}{n} \sum_{i=1}^n x_i^{\otimes d} - \frac{1}{m(\alpha)} \sum_{i=1}^m u_i^{\otimes d}}_2^2
  \lesssim 16^d \alpha,
  $$ 
  for all $d \in \N$. The  run time of the algorithm is $ C(\alpha)(np+q)$, where $q$ is the complexity to find a best approximation element of $K$ to a given vector in $\R^p$.
\end{theorem}

\section*{Acknowledgement}

M.B. acknowledges support from NSF DMS-2140592. T.S. acknowledges support from NIH R01HL16351, NSF DMS-2027248, and NSF DMS-2208356.
 R.V. acknowledges support from NSF DMS-1954233, NSF DMS-2027299, U.S. Army 76649-CS, and NSF+Simons Research Collaborations on the Mathematical and Scientific Foundations of Deep Learning.

%\bibliographystyle{plain}
%\bibliography{../machine}

\begin{thebibliography}{10}

\bibitem{bodwin2022unified}
Greg Bodwin and Santosh Vempala.
\newblock A unified view of graph regularity via matrix decompositions.
\newblock {\em Random Structures \& Algorithms}, 61(1):62--83, 2022.

\bibitem{BSV2021a}
March Boedihardjo, Thomas Strohmer, and Roman Vershyin.
\newblock {Covariance's Loss is Privacy's Gain:} {C}omputationally {E}fficient,
  {P}rivate and {A}ccurate {S}ynthetic {D}ata.
\newblock {\em Foundations of Computational Mathematics}, to appear.

\bibitem{deshpande2012zero}
Amit Deshpande, Ravindran Kannan, and Nikhil Srivastava.
\newblock Zero-one rounding of singular vectors.
\newblock In {\em International Colloquium on Automata, Languages, and
  Programming}, pages 278--289. Springer, 2012.

\bibitem{foffano2019you}
Daniele Foffano, Luca Rossi, and Andrea Torsello.
\newblock You can't see me: Anonymizing graphs using the {Szemeredi} regularity
  lemma.
\newblock {\em Frontiers in big Data}, 2:7, 2019.

\bibitem{frieze1999quick}
Alan Frieze and Ravi Kannan.
\newblock Quick approximation to matrices and applications.
\newblock {\em Combinatorica}, 19(2):175--220, 1999.

\bibitem{jain2022optimal}
Vishesh Jain, Ashwin Sah, and Mehtaab Sawhney.
\newblock Optimal minimization of the covariance loss.
\newblock {\em arXiv preprint arXiv:2205.01773}, 2022.

\bibitem{komlos2000regularity}
J{\'a}nos Koml{\'o}s, Ali Shokoufandeh, Mikl{\'o}s Simonovits, and Endre
  Szemer{\'e}di.
\newblock The regularity lemma and its applications in graph theory.
\newblock {\em Summer school on theoretical aspects of computer science}, pages
  84--112, 2000.

\bibitem{lovasz2007szemeredi}
L{\'a}szl{\'o} Lov{\'a}sz and Bal{\'a}zs Szegedy.
\newblock Szemer{\'e}di’s lemma for the analyst.
\newblock {\em GAFA Geometric And Functional Analysis}, 17(1):252--270, 2007.

\bibitem{minello2020k}
Giorgia Minello, Luca Rossi, and Andrea Torsello.
\newblock k-anonymity on graphs using the {S}zemer{\'e}di regularity lemma.
\newblock {\em IEEE Transactions on Network Science and Engineering},
  8(2):1283--1292, 2020.

\bibitem{szemeredi1975regular}
Endre Szemer{\'e}di.
\newblock Regular partitions of graphs.
\newblock In J.-C. Bermond, J.-C. Fournier, M.~Las Vergnas, and D.~Sotteau,
  editors, {\em Proc. Colloque Inter. CNRS}, page 399–401, 1978.

\bibitem{vershyninbook}
Roman Vershynin.
\newblock {\em High-dimensional probability. An introduction with applications
  in data science}.
\newblock Cambridge University Press, 2018.

\end{thebibliography}

\end{document}